\documentclass[reqno]{amsart}
\usepackage[
    bookmarksopen=true,
    citecolor=blue]{hyperref}
\usepackage{multirow,makecell}
\usepackage{float}
\usepackage{caption}
\usepackage[all]{xy}
\usepackage{eufrak}
\usepackage{mathrsfs, amscd}
\usepackage{amsmath}
\usepackage{amssymb}
\newtheorem{thm}{Theorem}[section]

\newtheorem{lem}{Lemma}[section]
\newtheorem{prop}{Proposition}[section]

\newtheorem{definition}{Definition}[section]
\newtheorem{df}{Definition}[section]

\newtheorem{remark}{Remark}[section]

\newcommand{\nc}{\newcommand}
\nc{\RBA}{{\mathrm{RBA}_\lambda}}
\nc{\sta}{\star_{A}}
\nc{\stb}{\star_{B}}
\nc{\s}{\star}
\nc{\C}{\mathrm{C}}
\nc{\ca}{\C_{\mathrm{Alg}}}
\nc{\crb}{\C_{\mathrm{RBO}_\lambda}}
\nc{\cra}{\C_{\mathrm{RBA}_\lambda}}
\nc{\cm}{\C_{\mathrm{mor}_\lambda}}
\nc{\hm}{\mathrm{H}_{\mathrm{mor}_\lambda}}

\allowdisplaybreaks[3]

\newcommand{\field}{\mathbb{F}}

\newcommand{\Hom}{{\rm Hom}}

\newcommand{\End}{{\rm End}}

\newcommand{\Hr}{{\rm H}}
\newcommand{\Zr}{{\rm Z}}
\newcommand{\Br}{{\rm B}}

\newcommand{\Id}{\rm Id}

\let \t=\otimes

{}

\begin{document}

\title[On the Cohomology of an associative Rota-Baxter operator morphism]{On the Cohomology of an  associative $\mathcal{O}$-operator morphism}



\author[Du]{Lei Du}
\address{School of Mathematical Sciences, Anhui University, Hefei, 230601, China}
\curraddr{}
\email{18715070845@163.com}
\author[Bao]{Yanhong Bao}
\address{School of Mathematical Sciences, Anhui University, Hefei, 230601, China}
\curraddr{}
\email{baoyh@ahu.edu.cn}
\author[Fu]{Dongxing Fu}
\address{School of Mathematical Sciences, Anhui University, Hefei, 230601, China}
\curraddr{}
\email{fudongxing920811@126.com}

\thanks{L. Du is the corresponding author. }



\date{}

\dedicatory{}

\begin{abstract}
Rota-Baxter operators and more generally $\mathcal{O}$-operators play a crucial role in broad areas of mathematics and physics, such as integrable systems, the Yang-Baxter equation and pre-Lie algebras. The main objects of study in the paper are certain $\mathcal{O}$-operator morphisms on associative algebras. The cohomology theory of an associative $\mathcal{O}$-operator morphism is  established. In development, we give the Cohomology Comparison Theorem of an $\mathcal{O}$-operator morphism, that is, the cohomology of an $\mathcal{O}$-operator morphism is isomorphic to the cohomology of an auxiliary $\mathcal{O}$-operator. As applications, we also study the Cohomology Comparison Theorem of a Rota-Baxter operator morphism (of weight zero) and an associative $r$-matrice weak morphism as
a particular case of $\mathcal{O}$-operator morphisms.
\end{abstract}

\maketitle

\section{Introduction}
Cohomology theories  have been developed with a great success on associative algebras \cite{GS0,GH}, Lie algebras \cite{CE} and group algebras \cite{SS}, which enable to control deformations problems.  The
idea of deforming a morphism of complex manifolds (in the special case where the target is fixed) has been considered by Kodaira \cite{K1, K2}.  Gerstenhaber and Schack \cite{ms1,ms2} develop a cohomology theory of algebra  morphisms to study  deformations of algebra morphisms, where a powerful result called the Cohomology
Comparison Theorem (CCT) is established.

The notion of $\mathcal{O}$-operators known as relative Rota-Baxter operators on associative algebras is a generalization of Rota-Baxter operators in the presence of bimodules. Rota-Baxter operators on associative algebras is introduced in 1960 by Baxter \cite{Bax1} to study a fluctuation theory in probability. Further, the relationship between Rota-Baxter operators and combinatorics is developed by Rota \cite{Rota1, Rota2}. Rota-Baxter operators have been showed many applications, for instance,   Connes-Kreimer’s  algebraic approach to the renormalization in perturbative quantum field theory \cite{Cn}, Yang-Baxter equation \cite{Bai1}. $\mathcal{O}$-operators is introduced by Uchino \cite{Uch} as an associative analogue of Poisson structures on a manifold. In \cite{Agu}, the author gives a class of interesting $\mathcal{O}$-operators induced by associative $r$-matrices. The cohomology theory of $\mathcal{O}$-operators had been absent for
a long time. Not until  recently has  the cohomology theory of $\mathcal{O}$-operators been concerned to study their deformations in \cite{Das, Tang}.

Our main objects of study in the paper are certain $\mathcal{O}$-operators on associative algebras. Exactly, we focus on cohomologies of associative $\mathcal{O}$-operator morphisms. Also, the Cohomology Comparison Theorem of $\mathcal{O}$-operator morphisms is given, which says that the cohomology of an $\mathcal{O}$-operator morphism is isomorphic to the cohomology of an auxiliary $\mathcal{O}$-operator. It has  been known earlier that Rota-Baxter operators of weight $0$ and associative $r$-matrices  \cite{Agu} are special cases of $\mathcal{O}$-operators. Therefore, as applications, we obtain the Cohomology Comparison Theorem (CCT) of  Rota-Baxter operator morphisms of weight zero and associative $r$-matrice weak morphisms.

The paper is organized as follows.  Section \ref{sect: O-operator}  provides some definitions of Rota-Baxter operators, $\mathcal{O}$-operators on associative algebras and their morphisms. Section \ref{sect: cohomology} sets up a cohomology theory for associative $\mathcal{O}$-operator morphisms and studies the Cohomology Comparison Theorem (CCT) of associative $\mathcal{O}$-operator morphisms, which says a cohomology of an associative $\mathcal{O}$-operator morphism is isomorphic to the cohomology of an auxiliary $\mathcal{O}$-operator. Section \ref{sect: applications} specializes to Rota-Baxter operators of weight $0$ on an associative  algebra $A$, regarded as $\mathcal{O}$-operators on $A$ with respect to the adjoint bimodule and focus on cohomologies of  $r$-matrice weak morphisms on an algebra
$A$, regarded as $\mathcal{O}$-operators on $A$ with respect to the coadjoint bimodule, thus the Cohomology Comparison Theorem of Rota-Baxter operator morphisms of weight zero and associative $r$-matrice weak morphisms are obtained.

In this paper, we work over a field $\field$ of characteristic 0 and unless otherwise specified, linear spaces, linear maps, $\otimes$, $\Hom$, $\End$ are defined over $\field$.

\section{$\mathcal{O}$-operators and their morphisms}\label{sect: O-operator}
In this section, we will recall Rota-Baxter operators, $\mathcal{O}$-operators on associative algebras and their morphisms. We first recall some basic concepts.

Let $(A, \cdot)$ be an associative algebra over  $\field$ and $M$ be a bimodule of $(A, \cdot)$. That is, there are linear maps $l : A \otimes M \rightarrow M, (a, m)\rightarrow l(a, m)$ and $r : M \otimes A \rightarrow M, (m, a) \rightarrow r(m, a)$ such that
$$l(a\cdot b, m)=l(a, l(b, m)), r(a,r(m,b))=r(l(a, m), b) \text{~and~} r(m, a\cdot b)=r(r(m,a), b),$$
for all $a, b\in A$. Thus, for each $a\in A$, there are maps $l_a : M \rightarrow M, m\rightarrow
l(a, m) \text{~and~} r_a : M\rightarrow M, m \rightarrow r(m, a)$. Sometimes we will write $am$ instead of $l(a, m)$ and $ma$ instead of $r(m, a)$
when there are no confusions.

It follows that the associative algebra $A$ is a bimodule over itself with the left and right
actions are given by the multiplication of $A$. We call this bimodule as adjoint bimodule.
The maps $l_a$ and $r_a$ for this bimodule are denoted by $ad^l_a$ and $ad^r_a$. The dual space $A^*$ also carries an $A$-bimodule (called coadjoint bimodule) structure with

$$l(a, f)(b) = f(b\cdot a) \text{~and~} r(f, a)(b) = f(a\cdot b),$$
for $a, b \in A$ and $f \in A^*$. The maps $l_a$ and $r_a$ for this bimodule are respectively denoted by $ad^{*l}_a$ and $ad^{*r}_a$.

\begin{df}\label{Def: Rota-Baxter algebra}{\rm (\cite{Fard-Guo,Guo1})}
Let $A$ be an associative algebra over  $\field$ and
$\lambda\in \field$. A linear
operator $R: A\rightarrow A$ is said to be a Rota-Baxter operator of
weight $\lambda$ if it satisfies
\begin{eqnarray}\label{Eq: Rota-Baxter relation}
R(a)R(b)=R\big(a R(b)+R(a) b+\lambda\  a b\big)\end{eqnarray}	
for any $a,b \in A$. In this case,  $(A,  R)$ is called a Rota-Baxter
algebra of weight $\lambda$.
\end{df}

The  concept of $\mathcal{O}$-operators (also called relative Rota-Baxter operators) are a generalization of Rota-Baxter operators in the presence of arbitrary bimodule.

\begin{df}\label{Def: operator}{\rm (\cite{Loday})}
Let $M$ be a bimodule of an associative algebra $A$. An $\mathcal{O}$-operator on $A$ with respect to the bimodule $M$ is  a linear map
$T_A: M\rightarrow A$ such that
\begin{align}
T_A(m)T_A(n)= T_A(mT_A(n)+T_A(m)n), \text{~for~all~} m, n\in M.\label{Eq: O-operator}
\end{align}
\end{df}

When $l_a$ and $r_a$ is the adjoint representation of $A$ denoted by $ad^l_a$ and $ad^r_a$, Eq. (\ref{Eq: O-operator}) reduces to Eq. (\ref{Eq: Rota-Baxter relation}) with
$\lambda = 0$, which means that a Rota-Baxter operator of weight 0 is an $\mathcal{O}$-operator on
$A$ with respect to the adjoint representation.

We refer to \cite{Das} to give the following results.
\begin{lem}\label{lem:star product}
Let $T_A: M\rightarrow A$ be an $\mathcal{O}$-operator. Then $M$ carries an associative algebra structure with the
product
\begin{align}
m \star_M m^\prime = mT_A(m^\prime) + T_A(m)m^\prime, \text{~for~} m, m^\prime \in M.\label{Eq: star-product}
\end{align}
We denote $(M,\star_M)$ by $M_\star$. Moreover, define
\begin{align}
&l_{T_A}: M \otimes A \rightarrow A, ~~~~~~~ l_{T_A} (m, a) = T_A(m)\cdot a - T_A(ma), \label{Eq: left star module}\\
&r_{T_A}: A \otimes M \rightarrow A, ~~~~~~~ r_{T_A} (a, m) = a\cdot T_A(m) - T_A(am), \label{Eq: right star module}
\end{align}
for $m\in M$ and $a\in A$. Then $(A, l_{T_A}, r_{T_A})$ is a bimodule over $M_\star$.
\end{lem}
Next we recall a morphism between $\mathcal{O}$-operators.
\begin{df}\label{Def: operator morphisms}{\rm(\cite{Das})}
A pair $(\phi, \psi)$ of an algebra morphism $\phi: A\rightarrow B$ and a linear map $\psi: M\rightarrow N$ is a morphism of of $\mathcal{O}$-operators from $T_A: M\rightarrow A$ to $T_B: N\rightarrow B$ if $(\phi, \psi)$ satisfies
\begin{align}
T_B\circ \psi &=\phi\circ T_A,\label{Eq:def morphism (1)}\\
\phi(a)\psi(m)&=\psi(am),\label{Eq:def morphism (2)}\\
\psi(m)\phi(a)&=\psi(ma)\label{Eq:def morphism (3)}.
\end{align}
\end{df}

Then we have the following lemma
\begin{lem}\label{lem:star algebra morphism}
Let $(\phi, \psi)$ be an $\mathcal{O}$-operator morphism from $T_A: M\rightarrow A$ to $T_B: N\rightarrow B$. Then
 $\psi$ is an algebra morphism from $M_\star$ to $N_\star$ and $\psi$ is denoted by $\psi_\star$.
\end{lem}
\begin{proof}
Only to show that for any $m, m^{\prime} \in M$,
\begin{align*}
\psi(m\star m^{\prime})&\stackrel{\text{Eq.} (\ref{Eq: star-product})}{=}\psi(mT_A(m^{\prime})+T_A(m)m^{\prime})\\
&\stackrel{\text{Eq.} (\ref{Eq:def morphism (3)})}{=}\psi(m)\phi(T_A(m^{\prime}))+\phi(T_A(m))\psi(m^{\prime})\\
&\stackrel{\text{Eq.} (\ref{Eq:def morphism (1)})}{=}\psi(m)T_B(\psi(m^{\prime}))+T_B(\psi(m))\psi(m^{\prime})\\
&\stackrel{\text{Eq.} (\ref{Eq: star-product})}{=}\psi(m)\star \psi(m^{\prime}).
\end{align*}
\end{proof}

Recall the definition of a bimodule over an associative algebra morphism is given in Gerstenhaber and Schack \cite{ms1}. Let $\alpha:A\rightarrow A^{\prime}$ be a morphism of associative algebras, then a $\alpha$-bimodule is a triple $\langle M, M^{\prime}, \beta\rangle$ such
that $M$ is a  bimodule over $A$, $M^{\prime}$ is a  bimodule over $A^{\prime}$, and $\beta: M\rightarrow M^\prime$
is an $A$-bimodule morphism, where $M^{\prime}$ is considered as a bimodule over associative algebra $A$ in a natural way.

By Lemma \ref{lem:star product}, let $(A, l_{T_A}, r_{T_A})$ be a bimodule of $M_\star$ and $(B, l_{T_B}, r_{T_B})$ a bimodule of $N_\star$.
Certainly, $B$ is  a bimodule over associative algebra $M_\star$ by
\begin{align}
M_\star\otimes B\longrightarrow^{(\psi_\star, \Id)}N_\star\otimes B\longrightarrow^{l_{T_B}} B
\end{align}
and $\phi: A\rightarrow B$ is a $M_\star$-bimodule morphism, denote by $\phi_\star$.
Then we have the following observation, which may play an important role to study the cohomology theory of $\mathcal{O}$-operator morphism in the next section.
\begin{lem}\label{lem:star algebra bimodule}
Let $(\phi, \psi)$ be an $\mathcal{O}$-operator morphism of from $T_A: M\rightarrow A$ to $T_B: N\rightarrow B$. Then $\langle A, B, \phi_\star\rangle$ is a bimodule over  an algebra morphism  $\psi_\star: M_\star \rightarrow N_\star$.
\end{lem}

\section{The cohomology of an associative $\mathcal{O}$-operator morphism}\label{sect: cohomology}
In this section, our aim is to show  the cohomology theory of an associative $\mathcal{O}$-operator morphism and give the the Cohomology Comparison Theorem of an $\mathcal{O}$-operator morphism.

\subsection{The cohomology of an associative $\mathcal{O}$-operator morphism.} Recall the cohomology theory of an associative $\mathcal{O}$-operator given by Das \cite{Das}.
Let $T_A: M\rightarrow A$ be an $\mathcal{O}$-operator. By Lemma \ref{lem:star product}, we  obtain an $M_\star$-bimodule $(A,l_{T_A}, r_{T_A})$. Therefore, we may consider the corresponding Hochschild cohomology of $M_\star$ with coefficients in $(A,l_{T_A}, r_{T_A})$.
More precisely, we define
$$C^n(M_\star, A): =\Hom(M_\star^{\otimes n}, A), \text{~for~} n\geqslant 0$$

and the differential given by
\begin{align}
d_{M,A}(a)(m)&=l_{T_A}(m, a)-r_{T_A}(a, m)\label{Eq: 0 differential of O-operator}\\
&=T_A(m)\cdot a - T_A(ma)- a\cdot T_A(m) + T_A(am)  \text{~for~} a\in A=C^0(M_\star, A) \nonumber
\end{align}

and

\begin{align}
&(d_{M,A}f)(m_1,...,m_{n+1}) \label{Eq: n differential of O-operator}\\
=&T_A(m_1)\cdot f(m_2,...,m_{n+1}) - T_A(m_1f(m_2,...,m_{n+1}))\nonumber\\
&+\sum^n_{i=1}(-1)^i f(m_1,...,m_{i-1}, m_iT_A(m_{i+1}) + T_A(m_i)m_{i+1},...,m_{n+1})\nonumber\\
&+ (-1)^{n+1} f(m_1,...,m_n)\cdot T_A(m_{n+1})\nonumber\\
&-(-1)^{n+1}T_A(f(m_1,...,m_n)m_{n+1}).\nonumber
\end{align}

Denote the group of $n$-cocycles by $\Zr^n(M_\star, A)$ and the group of $n$-coboundaries
by $\Br^n(M_\star, A)$. The corresponding cohomology groups are defined by $\Hr^n(M_\star, A) = \Zr^n(M_\star, A)/\Br^n(M_\star, A), n \geqslant 0$.
\begin{definition}{\rm (\cite{Das})}\label{lem:cohomo of O-operator}
The cohomology groups $H^{n}(M_\star, A)$ is denoted as the $n$-cohomology of an $\mathcal{O}$-operator $T_A: M\rightarrow A$.
\end{definition}
Then we construct the cohomology of an $\mathcal{O}$-operator morphism. Let $(\phi, \psi)$ be an $\mathcal{O}$-operator morphism from $T_A: M\rightarrow A$ to $T_B: N\rightarrow B$. By Lemma \ref{lem:star algebra morphism} and Lemma \ref{lem:star algebra bimodule}, we have that
 $\psi_\star: M_\star\rightarrow N_\star$ is an algebra morphism and $(A, B, \phi)$ is the bimodule over $\psi_\star: M_\star\rightarrow N_\star$.

We may consider the corresponding Hochschild cochain $(C^{\bullet}(M_\star, A), d^{\bullet}_{M,A})$ of $M_\star$ with coefficients in $A$, the corresponding Hochschild cochain $(C^{\bullet}(N_\star, B), d^{\bullet}_{N,B})$ of $N_\star$ with coefficients in $B$ and the corresponding Hochschild cochain $(C^{\bullet}(M_\star, B), d^{\bullet}_{M,B})$ of $M_\star$ with coefficients in $B$, where in the latter $B$ is viewed as a $M_\star$-bimodule by virtue of $\psi_\star$. Take $X^{\bullet}=C^{\bullet}(M_\star, A)\oplus C^{\bullet}(N_\star, B)$ and $Y^{\bullet}= C^{\bullet}(M_\star, B)$, where the differential of $X^{\bullet}$ is $\delta^n_{X}=\begin{pmatrix}d^n_{M,A}&0\\0& d^n_{N,B})\end{pmatrix}$ and the differential of $Y^{\bullet}$ is $\delta^n_{Y}=d^n_{M,B}$. Let $\alpha\in C^n(M_\star, A)$ and $\beta\in C^n(N_\star, B)$, define a chain map $f^n: X^n\rightarrow Y^n$ by
\begin{align}\label{Eq: chain map 1}
f^n(\alpha, \beta)=\phi_\ast\alpha-\psi_\star^\ast\beta,
\end{align}
where $\phi_\ast: C^n(M_\star, A)\rightarrow C^n(M_\star, B)$ and $\psi_\star^\ast: C^n(N_\star, B)\rightarrow C^n(M_\star, B)$ are the morphisms induced by $\phi$ and $\psi_\star$, respectively.  $W^\bullet=X^\bullet \oplus Y^{\bullet-1}$  is a complex with coboundary operator
\begin{align}
\delta_{Z}^n=
\begin{pmatrix}
\delta^n_X&0\\
f^n&-\delta_Y^n
\end{pmatrix},
\end{align}
where $W^\bullet$ is known as the mapping cylinder of $f^\bullet$. Then $W^n=C^n(M_\star, A)\oplus C^n (N_\star, B)\oplus C^{n-1}(M_\star, B)$ is denoted by $C^n(\psi_\star, \phi)$ and its differential is
\begin{align}\label{Eq: chain map 2}
\delta_{W}^n(\alpha,\beta,\gamma)=(d_{M,A}^n\alpha, d_{N,B}^n\beta, \phi\alpha-\beta\psi_\star^{\otimes n}-d_{M,B}^{n-1}\gamma),
\end{align}
for $(\alpha, \beta, \gamma)\in C^n(M_\star, A)\oplus C^n (N_\star, B)\oplus C^{n-1}(M_\star, B)$.

\begin{definition}\label{def: cohomology operator morphism}
Denote  the space of n-cocyles(resp. n-coboundaries) of the complex $(W^\bullet, \delta_{W}^\bullet)$ by $Z_n(\psi_\star, \phi)$
(resp. $B_n(\psi_\star, \phi)$).
Then $H^\bullet(\psi_\star, \phi)=Z_n(\psi_\star, \phi)/B_n(\psi_\star, \phi)$  is defined as the cohomology of an $\mathcal{O}$-operator morphism $(\phi, \psi)$.
\end{definition}

\begin{prop}\label{prop: 1-cocycle}
Let $(\phi, \psi)$ be an $\mathcal{O}$-operator morphism from $T_A: M\rightarrow A$ to $T_B: N\rightarrow B$.
Then $(T_A, T_B, 0)$ is a $1$-cocycle on the cohomology of the $\mathcal{O}$-operator morphism $(\phi, \psi)$.
\end{prop}

\begin{proof}
For $m_1, m_2\in M$ and $n_1, n_2\in N$, we have
\begin{align*}
&d^1_{M,A}(T_A)\\
\stackrel{\text{Eq.} (\ref{Eq: n differential of O-operator})}{=}&
T_A(m_1)T_A(m_2)-T_A(m_1T_A(m_2))-T_A(m_1T_A(m_2)+T_A(m_1)m_2)\\
&+T_A(m_1)T_A(m_2)-T_A(T_A(m_1)m_2)\\
=&2(T_A(m_1)T_A(m_2)-T_A(m_1T_A(m_2)+T_A(m_1)m_2))\\
\stackrel{\text{Eq.} (\ref{Eq: O-operator})}{=}&0,
\end{align*}
and
\begin{align*}
&d^1_{N,B}(T_B)\\
\stackrel{\text{Eq.} (\ref{Eq: n differential of O-operator})}{=}&
T_B(n_1)T_B(n_2)-T_B(n_1T_B(n_2))-T_B(n_1T_B(n_2)+T_B(n_1)n_2)\\
&+T_B(n_1)T_B(n_2)-T_B(T_B(n_1)n_2)\\
=&2(T_B(n_1)T_B(n_2)-T_B(n_1T_B(n_2)+T_B(n_1)n_2))\\
\stackrel{\text{Eq.} (\ref{Eq: O-operator})}{=}&0.
\end{align*}
Then
\begin{align*}
&\delta_{W}^1(T_A, T_B, 0)\\
=&(d^1_{M,A}(T_A), d^1_{N,B},\phi T_A-T_B\psi_{\star}-d^{0}_{M,B}(0))\\
=&(0, 0, \phi T_A-T_B\psi_{\star})\\
\stackrel{\text{Eq.} (\ref{Eq:def morphism (1)})}{=}&
0.
\end{align*}
\end{proof}

\begin{remark}{\rm (Infinitesimal deformations.)}
Algebraic deformation theory is firstly described due to the seminal work of Gerstenhaber \cite{GS1}. Then Gerstenhaber and Schack \cite{ms1} develop a cohomology theory of algebraic morphisms to study  deformations of algebraic morphisms. In the Remark, we may use the above cohomologies to describe infinitesimal deformations of $\mathcal{O}$-operator morphisms.

Let $(\phi, \psi)$ be an $\mathcal{O}$-operator morphism from $T_A: M\rightarrow A$ to $T_B: N\rightarrow B$. Let  $T_{A,t}: M[[t]]\rightarrow A[[t]]$(resp.  $T_{B,t}: N[[t]]\rightarrow B[[t]]$) be an $\mathcal{O}$-operators on $A[[t]]$(resp. $B[[t]]$) with respect to the
bimodule $M[[t]]$(resp. $N[[t]]$). A deformation $(\phi_t, \psi_t)$ of $(\phi, \psi)$ is an $\mathcal{O}$-operator morphism from $T_{A,t}$ to $T_{B,t}$  such that $\phi_t= \phi_0+t\phi_1+t^2\phi_2+\cdot\cdot\cdot$ and $\psi_t= \psi_0+t\psi_1+t^2\psi_2+\cdot\cdot\cdot$, where every $(\phi_i,\psi_i) (n=0,1,2,\cdot\cdot\cdot)$ is an $\mathcal{O}$-operator morphism from $T_A$ to $T_B$. $(\text{Here }\phi_0=\phi,\psi_0=\psi.)$ Write $T_{A,t}=T_{A,0}+tT_{A,1}+t^2T_{A,2}+\cdot\cdot\cdot$ and $T_{B,t}=T_{B,0}+tT_{B,1}+t^2T_{B,2}+\cdot\cdot\cdot$.
We claim that the triple $(T_{A, 1}, T_{B, 1}, 0)$ is an element of $\Zr^1(\phi,\psi)$ and its cohomology class may be viewed as the infinitesimal of the deformation.

That $(\phi_t, \psi_t): T_{A, t}\rightarrow T_{B, t}$ is an $\mathcal{O}$-operator deduces that
\begin{enumerate}
\item[{\rm (1)}] $T_{A,t}(m_1) T_{A,t}(m_2) = T_{A,t}(m_1T_{A,t}(m_2) + T_{A,t}(m_1)m_2),  \text{for } m_1, m_2\in M$.
\item[{\rm (2)}] $T_{B,t}(n_1) T_{B,t}(n_2) = T_{B,t}(n_1T_{B,t}(n_2) + T_{B,t}(n_1)n_2),  \text{for } n_1, n_2\in M$.
\item[{\rm (3)}] $T_{B,t}\circ \psi_t =\phi_t\circ T_{A,t}$.
\end{enumerate}
Comparing first-order terms, we have
\begin{align}&T_A(m_1)T_{A,1}(m_2) + T_{A,1}(m_1) T_A(m_2) \nonumber\\
=& T_A(m_1T_{A,1}(m_2) + T_{A,1}(m_1)m_2) + T_{A,1}(m_1T_A(m_2) + T_A(m_1)m_2),  \text{for } m_1, m_2\in M. \label{Eq: T_A1}\\
&T_B(n_1)T_{B,1}(n_2) + T_{B,1}(n_1) T_B(n_2) \nonumber\\
=& T_B(n_1T_{B,1}(n_2) + T_{B,1}(n_1)n_2) + T_{B,1}(n_1T_B(n_2) + T_B(n_1)n_2),  \text{for } n_1, n_2\in N. \label{Eq: T_B1}\\
&T_{B,1}\circ \psi_0+ T_{B,0}\circ \psi_1 =\phi_1\circ T_{A,0}+\phi_0\circ T_{A,1}.\label{Eq: psi_1 phi_1}
\end{align}

Then Eq. {\rm (\ref{Eq: T_A1})} and Eq. {\rm (\ref{Eq: T_B1})} deduce that $d_{M,A}^1(T_{A,1})=0$ and $d_{N,B}^1(T_{B,1})=0$.

Due to  Eq. {\rm (\ref{Eq: psi_1 phi_1})} and $T_{B}\circ \psi_1=\phi_1\circ T_{A}$ inferred from $(\phi_1,\psi_1)$ being an $\mathcal{O}$-operator from $T_A$ to $T_B$, we have
$T_{B,1}\circ \psi=\phi\circ T_{A,1}$.

Therefore, $\delta_{W}^1(T_{A, 1}, T_{B, 1}, 0)=(d_{M,A}^1(T_{A,1}), d_{N,B}^1(T_{B,1}), \phi T_{A,1}-T_{B,1}\psi_{\star}-d^{0}_{M,B}(0))=0$,
i.e., $(T_{A, 1}, T_{B, 1}, 0)$ is $1$-cocycle.   \flushright{$\square$}
\end{remark}

\subsection{The Cohomology Comparison Theorem of an $\mathcal{O}$-operator morphism.} We will give the Cohomology Comparison Theorem (CCT) of  $\mathcal{O}$-operator morphisms, which may say that the cohomology of an $\mathcal{O}$-operator morphism is isomorphic to the cohomology of an auxiliary $\mathcal{O}$-operator.

We recall some results of the bimodule of an algebra morphism given by Gerstenhaber and Schack \cite{ms1, ms2}.
Let $\phi:A\rightarrow B$ be a morphism of associative algebras and  $\langle M, N, \psi\rangle$ be a $\phi$-bimodule.
Denote the cochain complex of $\phi:A\rightarrow B$ with coefficients in $\langle M, N, \psi\rangle$  by $C^\bullet(\phi,\psi)$, where $\C^n(\phi,\psi)=C^n(A,M)\oplus C^n(B,N)\oplus C^{n-1}(A,N), n\geqslant 1$.
\begin{lem}{\rm (\cite{ms1})}\label{lem:phi!}
Suppose that $\phi: A\rightarrow B$ is an algebra morphism  and $\langle M, N, \psi\rangle$ is a $\phi$-bimodule. Then
\begin{enumerate}
\item[{\rm (1)}]
$\phi!=A+B+B\phi$ is an associative algebra with
$a\cdot b=b\cdot a=\phi\cdot b=a\cdot \phi=\phi^2=0$ and $\phi\cdot a=\phi(a)\phi$, that is,  for any $a+b_1+b_2\phi, a^\prime+b^\prime_1+b_2^\prime\phi\in A+B+B\phi$,
\begin{align}
(a+b_1+b_2\phi)(a^\prime+b^\prime_1+b_2^\prime\phi)&=a a^\prime+b_1 b_1^\prime+b_2\phi\cdot a^\prime+b_1b_2^\prime\cdot\phi \label{Eq: phi product}\\
&=a a^\prime+b_1 b_1^\prime+(b_2\phi(a^\prime)+b_1b_2^\prime)\cdot\phi.\nonumber
\end{align}

\item[{\rm (2)}] $\psi!=M+N+N\phi$ is a $\phi!$-bimodule with $\phi m=\psi(m)\phi$, i.e., the bimodule struction is defined by
\begin{align}
(a+b_1+b_2\phi)(m+n_1+n_2\phi)=am+b_1n_1+(b_2\psi(m)+b_1n_2)\phi,\label{Eq:phi left module}\\
(m+n_1+n_2\phi)(a+b_1+b_2\phi)=ma+n_1b_1+(n_2\phi(a)+n_1b_2)\phi,\label{Eq:phi right module}
\end{align}
for  $a+b_1+b_2\phi\in A+B+B\phi$, $m+n_1+n_2\phi\in M+N+N\phi$.
\end{enumerate}
\end{lem}

We denote the corresponding Hochschild cochain  of $\phi!$ with coefficients in $\psi!$ by $\C^\bullet_{\mathrm{Alg}}(\phi!,\psi!)$.
The following lemma  is the Cohomology Comparison Theorem (CCT) of an associative algebra morphism given by Gerstenhaber and Schack \cite{ms2}.

\begin{lem}{\rm(\cite{ms2})}\label{lem:CCT of classic algebra}
Let $\phi:A\rightarrow B$ be an associative algebra morphism, $\langle M,N ,\psi\rangle$ be a bimodule of $\phi$. Define
$\tau^\bullet_\phi:\C^\bullet(\phi,\psi)\rightarrow \C^\bullet_{\mathrm{Alg}}(\phi!,\psi!)$ as follows:
for $f=(f^A,f^B,f^{AB})\in \C^n(\phi,\psi)$, $\tau^n_\phi f$ defined by
\begin{align*}
&\tau^n_\phi f|_{B^{\t n}}=f^B, \tau_\phi f|_{A^{\t n}}=f^A;\\
&\tau^n_\phi f(b\phi,a_2,\cdots,a_n)=f^B(b,\phi(a_2),\cdots,\phi(a_n))\phi +b f^{AB}(a_2,\cdots,a_n)\phi, \\
&\textmd{for}~ ~(b\phi,a_2,\cdots,a_n)\in B\phi\t A^{n-1};\\
&\tau^n_\phi f(b_1,\cdots,b_{r-1},b_r\phi,a_{r+1},\cdots,a_n)=
f^B(b_1,\cdots,b_{r-1},b_r,\phi(a_{r+1}),\cdots,\phi(a_n))\phi,\\
&\textmd{for}~ ~(b_1,\cdots,b_{r-1},b_r\phi,a_{r+1},\cdots,a_n)\in B^{r-1}\t B\phi\t A^{n-r};\\
&\tau^n_\phi f(x_1,\cdots,x_n)=0, \indent\indent otherwise.
\end{align*}
Then $\tau^\bullet_\phi:\C^\bullet(\phi,\psi)\rightarrow \C^\bullet_{\mathrm{Alg}}(\phi!,\psi!)$ is a quasi-isomorphism.
\end{lem}

By Lemma \ref{lem:star algebra morphism} and Lemma \ref{lem:star algebra bimodule}, for an $\mathcal{O}$-operator morphism $(\phi, \psi):T_A\rightarrow T_B$, we have $\psi_\star:M_\star\rightarrow N_\star$ is an algebra morphism and $\langle A, B, \phi_\star\rangle$ is a  bimodule of  an algebra morphism  $\psi_\star: M_\star \rightarrow N_\star$. Due to Lemma \ref{lem:phi!}, $(\psi_\star !=M_\star+N_\star+N_\star\psi_\star, \ast)$ is an associative algebra, which product $\ast$ is defined by
\begin{align}
&(m+n_1+n_2\psi_\star)\ast(m^\prime+n_1^\prime+n_2^\prime\psi_\star)\nonumber\\
=&m\star_M m^\prime+n_1\star_Nn_1^\prime+(n_2\star_N\psi_\star(m^\prime)+n_1\star_N n_2^\prime)\psi_\star\label{Eq: *-product}
\end{align}
for $m+n_1+n_2\psi_\star, m^\prime+n_1^\prime+n_2^\prime\psi_\star\in \psi_\star!$
and
$(\phi_\star !=A+B+B\psi_\star, l_\star,r_\star)$ is a bimodule of $\psi_\star !$, which module action is
\begin{align}
&l_\star(m+n_1+n_2\psi_\star, a+b_1+b_2\psi_\star)\nonumber\\
=&l_{T_A}(m,a)+l_{T_B}(n_1,b_1)+(l_{T_B}(n_2,\phi(a))+l_{T_B}(n_1,b_2)))\psi_\star,\label{Eq:left action l_star}\\
&r_\star(a+b_1+b_2\psi_\star, m+n_1+n_2\psi_\star)\nonumber\\
=&r_{T_A}(a, m)+r_{T_B}(b_1,n_1)+(r_{T_B}(b_2,\psi_\star(m))+r_{T_B}(b_1,n_2))\psi_\star\label{Eq:right action l_star}
\end{align}
for $m+n_1+n_2\psi_\star \in \psi_\star !$ and $a+b_1+b_2\psi_\star\in\phi_\star !$. On account of Lemma \ref{lem:CCT of classic algebra},  it follows that

\begin{lem}\label{lem: CCT used}
Let $(\phi, \psi)$ be an $\mathcal{O}$-operator morphism. Then there is a quasi-isomorphism between $C^\bullet(\psi_\star, \phi)$ and $C^\bullet(\psi_{\star}!, \phi_\star!)$, i.e., the cohomology $H^\bullet(\psi_\star, \phi)$ of an $\mathcal{O}$-operator morphism $(\phi, \psi)$ is isomorphic to the corresponding Hochschild cohomology $H^\bullet(\psi_{\star}!, \phi_\star!)$ of $\psi_\star!$ with coefficients in $\phi_\star!$.
\end{lem}

By Lemma \ref{lem:phi!}, it follows that the following lemma to construct a new $\mathcal{O}$-operator from an $\mathcal{O}$-operator morphism.

\begin{lem}\label{lem: new o-operator}
Let $(\phi, \psi)$ be an $\mathcal{O}$-operator morphism from $T_A: M\rightarrow A$ to $T_B: N\rightarrow B$. Suppose that $T_{\phi!}: \psi!\rightarrow \phi!$ is a linear map defined by
\begin{align}
T_{\phi!}(m+n_1+n_2\phi)=T_A(m)+T_B(n_1)+T_B(n_2)\phi\label{eq: T_phi}
\end{align}
for $m+n_1+n_2\phi\in\psi!$,
where $\psi!=M+N+N\phi$ and $\phi!=A+B+B\phi$.
Then $T_{\phi!}$ is an $\mathcal{O}$-operator.
\end{lem}
\begin{proof}
Only to show that $T_{\phi!}$ satisfies Eq. (\ref{Eq: O-operator}). For $m+n_1+n_2\phi, m^\prime+n^\prime_1+n^\prime_2\phi\in\psi!$,
\begin{align*}
&T_{\phi!}(m+n_1+n_2\phi)T_{\phi!}(m^\prime+n^\prime_1+n^\prime_2\phi)\\
\stackrel{\text{Eq.} (\ref{eq: T_phi})}{=}&(T_A(m)+T_B(n_1)+T_2(n_2)\phi)(T_A(m^\prime)+T_B(n_1^\prime)+T_2(n_2^\prime)\phi)\\
\stackrel{\text{Eq.} (\ref{Eq: phi product})}{=}& T_A(m)T_A(m^\prime)+T_B(n_1)T_B(n_1^\prime)+(T_B(n_1)T_B(n_2^\prime)+T_B(n_2)\phi(T_A(m^\prime)))\phi\\
\stackrel{\text{Eq.} (\ref{Eq:def morphism (1)})}{=}&
T_A(m)T_A(m^\prime)+T_B(n_1)T_B(n_1^\prime)+(T_B(n_1)T_B(n_2^\prime)+T_B(n_2)T_B(\psi(m^\prime)))\phi\\
\stackrel{\text{Eq.} (\ref{Eq: O-operator})}{=}&
T_A(mT_A(m^\prime)+T_A(m)m^\prime)+T_B(n_1T_B(n_1^\prime)+T_B(n_1)n_1^\prime)\\
&+(T_B(n_1T_B(n_2^\prime)+T_B(n_1)n_2^\prime)+T_B(n_2T_B(\psi(m^\prime))+T_B(n_2)\psi(m^\prime)))\phi\\
\stackrel{\text{Eq.} (\ref{eq: T_phi})}{=}&
T_{\phi!}(mT_A(m^\prime)+T_A(m)m^\prime+n_1T_B(n_1^\prime)+T_B(n_1)n_1^\prime+(n_1T_B(n_2^\prime)+T_B(n_1)n_2^\prime)\phi)\\
&+(n_2T_B(\psi(m^\prime))+T_B(n_2)\psi(m^\prime))\phi)\\
\stackrel{\text{Eq.} (\ref{Eq:def morphism (1)})}{=}&
T_{\phi!}(mT_A(m^\prime)+n_1T_B(n_1^\prime)+(n_2\phi(T_A(m^\prime)+n_1T_B(n_2^\prime))\phi\\ &+T_A(m)m^\prime+T_B(n_1)n_1^\prime+(T_B(n_1)n_2^\prime+T_B(n_2)\psi(m^\prime))\phi)\\
\stackrel{\text{Eq.} (\ref{Eq:phi left module})(\ref{Eq:phi right module})}{=}&
T_{\phi!}((m+n_1+n_2\phi)(T_A(m^\prime)+T_B(n_1^\prime)+T_B(n_2^\prime)\phi)\\
&+(T_A(m)+T_B(n_1)+T_2(n_2)\phi)(m^\prime+n^\prime_1+n^\prime_2\phi))\\
\stackrel{\text{Eq.} (\ref{eq: T_phi})}{=}&
T_{\phi!}((m+n_1+n_2\phi)T_{\phi!}(m^\prime+n^\prime_1+n^\prime_2\phi)+T_{\phi!}(m+n_1+n_2\phi)(m^\prime+n^\prime_1+n^\prime_2\phi)).
\end{align*}
\end{proof}

By Lemma \ref{lem:star product}, the $\mathcal{O}$-operator $T_{\phi!}: \psi!\rightarrow \phi!$ given by Lemma \ref{lem: new o-operator}, may induce an associative algebra $(\psi!_\star, \star_{\psi!})$ and the product $\star_{\psi!}$ means that
\begin{align}
&(m+n_1+n_2\phi)\star_{\psi!}(m^\prime+n_1^\prime+n_2^\prime\phi)\nonumber\\
\stackrel{\text{Eq.} (\ref{Eq: star-product})}{=}&
(m+n_1+n_2\phi)T_{\phi!}(m^\prime+n_1^\prime+n_2^\prime\phi)+T_{\phi!}(m+n_1+n_2\phi)(m^\prime+n_1^\prime+n_2^\prime\phi).\label{Eq: !star-product}
\end{align}

and a bimodule  $(\phi!_\star, l_{T_{\phi!}}, r_{T_{\phi!}})$ over $\psi!_\star$, where $l_{T_{\phi!}}, r_{T_{\phi!}}$ is defined by

\begin{align}
&l_{T_{\phi!}}(m+n_1+n_2\phi, a+b_1+b_2\phi) \nonumber\\
=&T_{\phi!}(m+n_1+n_2\phi)(a+b_1+b_2\phi)-T_{\phi!}((m+n_1+n_2\phi)(a+b_1+b_2\phi), \label{Eq: left_T_phi}\\
&r_{T_{\phi!}}(a+b_1+b_2\phi, m+n_1+n_2\phi) \nonumber\\
=&(a+b_1+b_2\phi)T_{\phi!}(m+n_1+n_2\phi)-T_{\phi!}((a+b_1+b_2\phi)(m+n_1+n_2\phi)). \label{Eq: right T phi}
\end{align}

By Definition \ref{lem:cohomo of O-operator},  the corresponding Hochschild cohomology $H^n(\psi!_\star, \phi!_\star)$ of $\psi!_\star$ with coefficients in $\phi!_\star$ is considered as the cohomology of $T_{\phi!}$.

In the last of this section, we may show that the cohomology of an $\mathcal{O}$-operator morphism $(\phi, \psi)$ is isomorphic to the cohomology of the $\mathcal{O}$-operator $T_{\phi!}$ induced by $(\phi, \psi)$. For this, we need the following lemma

\begin{lem}\label{lem: Psi iso}
Let $(\phi,\psi)$ be an $\mathcal{O}$-operator from $T_A: M\rightarrow A$ to $T_B: N\rightarrow B$. Then there exists an algebra isomorphism
$\Psi:  (\psi!_{\star},\star_{\psi!})\rightarrow  (\psi_{\star}!, \ast)$ defined by for $m+n_1+n_2\phi\in \psi!_{\star}$ and $m+n_1+n_2\psi_\star \in \psi_{\star}!$,
\begin{align*}
\Psi (m+n_1+n_2\phi)= m+n_1+n_2\psi_\star,
\end{align*}
where $(\psi!_{\star},\star_{\psi!})$ and $(\psi_{\star}!, \ast)$ are defined by Eq.{\rm (\ref{Eq: !star-product})} and Eq. {\rm (\ref{Eq: *-product})}.
\end{lem}
\begin{proof}
Clearly, $\Psi$ is a bijection.
We only to show that $\Psi$ is  an algebraic morphism.
By Eq. {\rm(\ref{Eq: !star-product})} and Eq. {\rm (\ref{Eq: *-product})}, we have
\begin{align}
&(m+n_1+n_2\phi)\star_{\psi!}(m^\prime+n_1^\prime+n_2^\prime\phi)\nonumber\\
\stackrel{\text{Eq.} (\ref{Eq: star-product})}{=}&
(m+n_1+n_2\phi)T_{\phi!}(m^\prime+n_1^\prime+n_2^\prime\phi)+T_{\phi!}(m+n_1+n_2\phi)(m^\prime+n_1^\prime+n_2^\prime\phi)\nonumber\\
\stackrel{\text{Eq.} (\ref{eq: T_phi})\text{Eq.}(\ref{Eq: phi product})}{=}&
mT_A(m^\prime)+n_1T_B(n^\prime_1)+(n_2\phi(T_A(m^\prime))+n_1T_B(n_2^\prime))\phi\nonumber\\
&+T_A(m)m^\prime+T_B(n_1)n_1^\prime+(T_B(n_1)n_2^\prime+T_B(n_2)\psi(m^\prime))\phi.\label{Eq:lem 3.5(1)}
\end{align}
and
\begin{align}
&(m+n_1+n_2\psi_\star)\ast(m^\prime+n_1^\prime+n_2^\prime\psi_\star)\nonumber\\
\stackrel{\text{Eq.} (\ref{Eq: *-product})}{=}&
m\star_M m^\prime+n_1\star_Nn_1^\prime+(n_2\star_N\psi_\star(m^\prime)+n_1\star_N n_2^\prime)\psi_\star \nonumber\\
\stackrel{\text{Eq.} (\ref{Eq: star-product})\text{Eq.} (\ref{Eq:def morphism (1)})}{=}&
mT_A(m^\prime)+T_A(m)m^\prime+n_1T_B(n_1^\prime)+T_B(n_1)n_1^\prime \nonumber\\
&+(n_2\phi(T_A(m^\prime))+n_1T_B(n_2^\prime))\psi_\star+(T_B(n_1)n_2^\prime+T_B(n_2)\psi_\star(m^\prime))\psi_\star. \label{Eq:lem 3.5(2)}
\end{align}
Comparing Eq. {\rm (\ref{Eq:lem 3.5(1)})} with Eq. {\rm (\ref{Eq:lem 3.5(2)})}, we have that $\Psi$ is an algebraic morphism.
\end{proof}
In development, it follows that by Lemma \ref{lem: Psi iso}
\begin{lem}\label{lem: iso bimodule star!}
$\psi!_\star$-bimodule $(\phi!_\star, l_{T_{\phi!}}, r_{T_{\phi!}})$ is equal to
$\psi_{\star}!$-bimodule $(\phi_\star !, l_\star,r_\star)$.
\end{lem}

\begin{proof}
To construct a linear map $\Phi: \phi!_\star\rightarrow \phi_\star !$ defined by $$\Phi(a+b_1+b_2\phi)=a+b_1+b_2\psi_\star$$
for $a+b_1+b_2\phi\in \phi!_\star$ and $a+b_1+b_2\psi_\star\in \phi_\star !$.  $\Phi$ is obviously bijective.
By Eq. {\rm (\ref{Eq: left_T_phi})} and  Eq. {\rm (\ref{Eq:left action l_star})}, we have
\begin{align}
&l_{T_{\phi!}}(m+n_1+n_2\phi, a+b_1+b_2\phi) \nonumber\\
\stackrel{\text{Eq.} (\ref{Eq: left_T_phi})}{=}&
T_{\phi!}(m+n_1+n_2\phi)(a+b_1+b_2\phi)-T_{\phi!}((m+n_1+n_2\phi)(a+b_1+b_2\phi) \nonumber\\
\stackrel{\text{Eq.} (\ref{eq: T_phi}) \text{Eq.} (\ref{Eq: phi product})}{=}&
T_A(m)a+T_B(n_1)b_1+(T_B(n_1)b_2+T_B(n_2)\phi(a))\phi \nonumber\\
&-T_A(ma)-T_B(n_1b_1)-(T_B(n_2\phi(a))+T_B(n_1b_2))\phi \label{Eq:lem proof 1}
\end{align}
and
\begin{align}
&l_\star(m+n_1+n_2\psi_\star, a+b_1+b_2\psi_\star)\nonumber\\
\stackrel{\text{Eq.} (\ref{Eq:left action l_star})}{=}&
l_{T_A}(m,a)+l_{T_B}(n_1,b_1)+(l_{T_B}(n_2,\phi(a))+l_{T_B}(n_1,b_2)))\psi_\star \nonumber\\
\stackrel{\text{Eq.} (\ref{Eq: left star module})}{=}&
T_A(m)a-T_A(ma)+T_B(n_1)b_1-T_B(n_1b_1)\nonumber\\
&+(T_B(n_1)b_2+T_B(n_2)\phi(a))\psi_\star-(T_B(n_2\phi(a))+T_B(n_1b_2))\psi_\star \label{Eq:lem proof 2}.
\end{align}
By Eq. {\rm (\ref{Eq:lem proof 1})} and Eq. {\rm (\ref{Eq:lem proof 2})}, it follows that
\begin{align*}
&\Phi(l_{T_{\phi!}}(m+n_1+n_2\phi, a+b_1+b_2\phi)) \\
=&l_\star(m+n_1+n_2\psi_\star, a+b_1+b_2\psi_\star)\\
=&l_\star((\Psi, \Phi)(m+n_1+n_2\phi, a+b_1+b_2\phi)).
\end{align*}
Similarly, $\Phi\circ r_{T_{\phi!}}= r_\star\circ (\Phi, \Psi)$.
\end{proof}

By Lemma \ref{lem: iso bimodule star!}, we identify the Hochschild cochain $C^\bullet(\psi_{\star}!, \phi_\star!)$ with the Hochschild cochain
$C^\bullet(\psi!_\star, \phi!_\star)$.
From Lemma \ref{lem: CCT used}, we get the following main result, which  is called the Cohomology Comparison Theorem (CCT) of an $\mathcal{O}$-operator morphism.
\begin{thm}\label{thm: CCT of O-operator}
Let $(\psi, \phi)$ be an $\mathcal{O}$-operator morphism from $T_A: M\rightarrow A$ to $T_B: N\rightarrow B$. Then the cohomology $H^\bullet(\psi_\star, \phi)$ of an $\mathcal{O}$-operator morphism $(\phi, \psi)$ is isomorphic to the cohomology $H^\bullet(\psi!_\star, \phi!_\star)$ of an $\mathcal{O}$-operator $T_{\phi!}: \psi!\rightarrow \phi!$.
\end{thm}

\section{Applications}\label{sect: applications}
It was shown  that $\mathcal{O}$-operators generalize associative Rota-Baxter
operators of weight $0$. They also generalize associative r-matrices \cite{Agu}. Therefore, we may
study the Cohomologies Comparison Theorem (CCT) of  Rota-Baxter operator morphisms of weight zero and associative $r$-matrice morphisms as
a particular case of cohomologies of $\mathcal{O}$-operator morphisms.

\subsection{The Cohomologies Comparison Theorem of  Rota-Baxter operator morphisms.}
As is well-known, a Rota-Baxter operator (of weight $0$) on an associative algebra $A$ can be considered as an $\mathcal{O}$-operator on $A$ with respect to the adjoint bimodule $A$.

Given a Rota-Baxter operator $R: A\rightarrow A$ (of weight $0$), $A$ carries a new associative product $a\star b =aR(b)+R(a)b$  for $a, b\in A$. This associative algebra $(A, \star)$ has a bimodule representation on $A$ given by $l_a(b) = R(a) b - R(a b)$ and $r_a(b) = b R(a) -R(b  a)$. The cohomology of the associative algebra $(A, \star)$ with coefficients in the above bimodule
structure on $A$ is called the cohomology of the Rota-Baxter operator $R$.

Let $R_A: A\rightarrow A$ and $R_B:B\rightarrow B$ be  Rota-Baxter operators of weight $0$. Suppose that $\phi: R_A\rightarrow R_B$ is a Rota-Baxter operator morphism, where $\phi: A\rightarrow B$ is an algebraic morphism such that (See Definition \ref{Def: operator morphisms})
\begin{align}
R_B\circ \phi &=\phi\circ R_A \label{Eq: RB-operator}.
\end{align}

Since $A$ is the adjoint bimodule over itself, $C^n(\phi_\star, \phi)=C^n(A_\star, A)\oplus C^n(B_\star,B)\oplus C^{n-1}(A_\star,B), n\geqslant 1$ is defined as a cochain complex of a Rota-Baxter operator morphism $\phi: R_A\rightarrow R_B$, which differential is
\begin{align}\label{Eq: chain map 2}
\delta_{R}^n(f, g,h)=(d_{A}^nf, d_{B}^n g, \phi f-g \phi_\star^{\otimes n}-d_{A,B}^{n-1}h),
\end{align}
for $(f, g, h)\in C^n(A_\star, A)\oplus C^n (B_\star, B)\oplus C^{n-1}(A_\star, B)$. Then the  corresponding cohomology $H^\bullet(\phi_\star, \phi)$ is defined as the cohomology of a Rota-Baxter operator morphism $\phi$. (See Definition \ref{def: cohomology operator morphism})

Given a Rota-Baxter operator morphism $\phi: R_A\rightarrow R_B$,  we have a new Rota-Baxter operator (of weight $0$) $R_{\phi !}: \phi !\rightarrow \phi!$, where $R_{\phi !}(a+b_1+b_2\phi)=R_A(a)+R_B(b_1)+R_B(b_2)\phi, a+b_1+b_2\phi \in \phi!= A+B+ B\phi$ (See Lemma \ref{lem: new o-operator}). Then it follows that by Theorem \ref{thm: CCT of O-operator}
\begin{prop}
Suppose that $\phi: R_A\rightarrow R_B$ is a Rota-Baxter operator morphism, where $R_A: A\rightarrow A$ and $R_B:B\rightarrow B$ are Rota-Baxter operators of weight $0$. Then the cohomology of a Rota-Baxter operator morphism $\phi$ is isomorphic to the cohomology of
$R_{\phi !}: \phi !\rightarrow \phi!$.
\end{prop}

\subsection{The Cohomologies Comparison Theorem of associative $r$-matrices weak morphisms.}

Aguiar\cite{Agu} introduces the notion of associative $r$-matrix as an associative analogue of classical $r$-matrix. An associative r-matrix can be considered as an $\mathcal{O}$-operator.

Let $A$ be an associative algebra and $r\in \wedge^2 A$, where $r$ induces a skew-symmetric linear map $r^{\sharp}: A^\ast\rightarrow A$ defined by
\begin{align}
\langle \nu, r^\sharp(\mu)\rangle=r(\mu, \nu), \mu, \nu\in A^\ast,
\end{align}
where $\langle , \rangle$ denotes the pairing between the elements of $A$ and $A^\ast$.

\begin{definition}
Let $A$ be an associative algebra. Then $r\in \wedge^2 A$ is called an associative $r$-matrix if $r$ satisfies $[[r, r]] = 0$, where
$[[r, r]] \in A \otimes A \otimes A$ is given by
\begin{align*}
[[r, r]](\mu, \nu, \omega)=\langle r^\sharp(\mu)r^\sharp(\nu), \omega\rangle + \langle r^\sharp(\nu)r^\sharp(\omega), \mu\rangle
+\langle r^\sharp(\omega)r^\sharp(\mu), \nu\rangle,
\end{align*}
for $\mu, \nu, \omega\in A^\ast$.
\end{definition}

The following result tells us the relation between associative $r$-matrixs and $\mathcal{O}$-operators.

\begin{prop}{\rm (\cite{Bai2})}\label{prop: O-operator and r-matrice}
$r\in \wedge^2 A$ is an associative $r$-matrix if and only if the induced
map $r^\sharp : A^\ast\rightarrow A$ is an $\mathcal{O}$-operator on $A$ with respect to the coadjoint $A$-bimodule $A^\ast$.
\end{prop}

Let $A$ be an associative algebra and $r_1, r_2 \in \wedge^2 A$ be two $r$-matrices.
\begin{definition}
A weak morphism from $r_1$ to $r_2$ consists of a pair $(\phi, \psi)$ of an associative algebra
morphism $\phi: A \rightarrow A$ and a linear map $\psi : A \rightarrow A$ satisfying
\begin{align}
(\psi \otimes \Id_A)(r_2)&=(\Id_A\otimes \phi)(r_1),\nonumber\\
\psi(\phi(a)b)&=a\psi(b),\nonumber\\
\psi(a\phi(b))&=\psi(a)b.\nonumber\\
\end{align}
\end{definition}

The weak morphism of associative $r$-matrices is also related to the morphism between corresponding $\mathcal{O}$-operators in \cite{Das} as follows.

\begin{prop}
Let $r_1, r_2 \in \wedge^2 A$ be two associative $r$-matrices on an associative algebra $A$. Then a pair $(\phi, \psi)$ is a weak morphism from $r_1$ to $r_2$ if and only if $(\phi, \psi^\ast)$ is an $\mathcal{O}$-operator morphism  from $r^\sharp_1$ to $r^\sharp_2$.
\end{prop}

Then we consider the cohomology of  an $\mathcal{O}$-operator morphism $(\phi, \psi^\ast)$ from $r^\sharp_1: A^*\rightarrow A$ to $r^\sharp_2: A^*\rightarrow A$  as the cohomology of a weak morphism $(\phi, \psi)$ from $r_1$ to $r_2$.

By Lemma \ref{lem: new o-operator}, there exists an $\mathcal{O}$-operator $r_{\phi !}^{\sharp}: \psi^\ast!\rightarrow \phi!$ defined by
\begin{align*}
r_{\phi !}^{\sharp}(\mu+\nu_1+\nu_2\phi)=r_1^\sharp(\mu)+r_2^\sharp(\nu_1)+r_2^\sharp(\nu_2)\phi
\end{align*}
for $\mu+\nu_1+\nu_2\phi \in \psi^\ast!$,
where $\psi^\ast!=A^*+A^*+A^*\phi$ and $\phi!=A+A+A\phi$. By Theorem \ref{thm: CCT of O-operator}, the cohomology of  the $\mathcal{O}$-operator $r_{\phi !}^{\sharp}$ is isomorphic to  the cohomology of an $\mathcal{O}$-operator morphism $(\phi, \psi^\ast)$ from $r^\sharp_1$ to $r^\sharp_2$.

If $\psi^\ast!=A^*+A^*+A^*\phi\triangleq (A+A+A\phi)^\ast=(\phi!)^\ast$ as $\phi!$-module, then  the $\mathcal{O}$-operator $r_{\phi !}^{\sharp}$ corresponds to  some  $r$-matrice $r_{\phi!}\in \phi!\wedge \phi!$ (See Proposition \ref{prop: O-operator and r-matrice}).  Therefore, we get the Cohomologies Comparison Theorem (CCT) of associative $r$-matrices weak morphisms.


\begin{thebibliography}{00}
\bibitem{Agu}
M. Aguiar. Infinitesimal Hopf algebras, in: New Trends in Hopf Algebra Theory, La Falda, 1999, in:
{\it Contemp. Math.}, vol. 267, Amer. Math. Soc., Providence, RI, 2000, pp. 1–29.

\bibitem{Bai1}
 C. Bai, A unified algebraic approach to classical Yang-Baxter equation. {\it J. Phys. A Math. Theor.} 40 (2007) 11073-11082.
\bibitem{Bai2}
C. Bai. Double constructions of Frobenius algebras, Connes cocycles and their duality, {\it J. Noncommut. Geom.}  4 (2010) 475-530.
\bibitem{Bax1}
G. Baxter. An analytic problem whose solution follows from a simple algebraic identity. {\it Pac. J. Math.} 10 (1960) 731–742.

\bibitem{CE} C. Chevalley and S. Eilenberg. Cohomology theory of Lie groups and Lie algebras. {\it Trans. Amer. Math. Soc.} 63 (1948), 85-124.

\bibitem{Cn}
A. Connes and D. Kreimer. Renormalization in quantum field theory and the Riemann-Hilbert problem.
I. The Hopf algebra structure of graphs and the main theorem. {\it Commun. Math. Phys.} 210(1) (2000) 249-273.

\bibitem{Das}
A. Das. Deformations of associative Rota-Baxter operators. {\it J. Algebra}  560 (2020), 144-180.
\bibitem{Fard-Guo}
K. Ebrahimi-Fard and L. Guo. Rota-Baxter algebras and dendriform algebras. {\it J. Pure Appl. Algebra}  212(2) (2008) 320-339.
\bibitem{GS0} M. Gerstenhaber. The cohomology structure of an associative ring. {\it Ann. of Math.} 78(2) (1963), 267-288.
\bibitem{GS1} M. Gerstenhaber. On the deformation of rings and algebras. {\it Ann. of Math.} 79 (1964) 59-103.
\bibitem{ms1}
M. Gerstenhaber and S. D. Schack. On the deformation of algebra morphisms and diagrams. {\it Trans. Amer. Math. Soc.}  279(1) (1983) 1-50.


\bibitem{ms2}
M. Gerstenhaber and S.D. Schack. On the cohomology of an algebra morphism. {\it J. Algebra}  95 (1985) 245-262.


\bibitem{Guo1}
L.~Guo. \emph{An Introduction to Rota-Baxter Algebra}.  International Press (US) and Higher
Education Press (China), 2012.

\bibitem{GH} G. Hochschild. On the cohomology groups of an associative algebra. {\it Ann. of Math.} 46(2) (1945), 58-67.

\bibitem{K1}
K. Kodaira. A theorem of completeness of characteristic systems for analytic families of compact
submanifolds of complex manifolds, {\it Ann. of Math.} 75(2) (1962) 146-162.

\bibitem{K2}
K. Kodaira. On stability of compact submanifolds of complex manifolds. {\it Amer. J. Math.} 85 (1963) 79-94.

\bibitem{Loday}
J.-L. Loday. Dialgebras, in: Dialgebras and Related Operads, in: Lecture Notes in Math., vol. 1763,
Springer, Berlin, 2001, pp. 7-66.

\bibitem{Rota1}
G.-C. Rota. Baxter algebras and combinatorial identities, I. {\it Bull. Am. Math. Soc.} 75 (1969) 325-329.

\bibitem{Rota2}
G.-C. Rota. Baxter algebras and combinatorial identities, II. {\it Bull. Am. Math. Soc.} 75 (1969) 330-334.

\bibitem{SS}S. F. Siegel and S. J. Witherspoon. The Hochschild cohomology ring of a group algebra. {\it Proc. London Math. Soc.} 79(1) (1999), 131-157.
\bibitem{Tang}
Tang, Rong; Bai, Chengming; Guo, Li and Sheng, Yunhe. Deformations and their controlling cohomologies of $\mathcal{O}$-operators. {\it Comm. Math. Phys.} 368(2) (2019) 665-700.

\bibitem{Uch}
K. Uchino. Quantum analogy of Poisson geometry, related dendriform algebras and Rota-Baxter operators. {\it Lett. Math. Phys.} 85(2–3) (2008) 91-109.




\end{thebibliography}
\end{document}